\documentclass[12pt]{amsart}

\usepackage{latexsym, amssymb, amsmath}

\usepackage{amsfonts, graphicx}
\usepackage{graphicx,color}
\newcommand{\be}{\begin{equation}}
\newcommand{\ee}{\end{equation}}
\newcommand{\beq}{\begin{eqnarray}}
\newcommand{\eeq}{\end{eqnarray}}

\newtheorem{thm}{Theorem}[section]

\newtheorem{lma}{Lemma}[section]

\newtheorem{cor}{Corollary}[section]

\theoremstyle{remark}
\newtheorem{rem}{Remark}[section]
\numberwithin{equation}{section}

\def\p{\partial}

\def\p{\partial}

\def\Ric{\text{\rm Ric}}

\def\Pi{\overline{\displaystyle{\mathbb{II}}}}

\def\a{\alpha}
\def\b{\beta}

\def\K{K\"ahler }

\def\Ric{\text{Ric}}

\def\a{\alpha}

\def\p{\partial}

\def\Ric{\operatorname{Ric}}
\def\Rm{\operatorname{Rm}}

\def\K{K\"ahler\ }
\def\be{\begin{equation}}
\def\ee{\end{equation}}

\def\b{\bar}

\def\cd{\nabla}

\def\bee
{\begin{equation*}}

\def\eee{\end{equation*}}

\def\bee{\begin{equation*}}
\def\eee{\end{equation*}}

\def\a{{\alpha}}
\def\b{{\beta}}

\def\p{\partial}

\def\K{K\"ahler }

\def\KRF{K\"ahler-Ricci flow }

\def\be{\begin{equation}}
\def\ee{\end{equation}}

\def\a{{\alpha}}
\def\b{{\beta}}

\def\Ric{\text{\rm Ric}}
\def\Rm{\text{\rm Rm}}

\def\p{\partial}

\def\p{\partial}

\def\p{\partial}

\def\KRF{K\"ahler-Ricci flow }

\def\cd{\nabla}

\def\bee{\begin{equation*}}
\def\eee{\end{equation*}}

\def\a{{\alpha}}
\def\b{{\beta}}

\def\p{\partial}

\def\K{K\"ahler }

\def\KRF{K\"ahler-Ricci flow }

\def\be{\begin{equation}}
\def\ee{\end{equation}}

\def\a{{\alpha}}
\def\b{{\beta}}

\def\Ric{\text{\rm Ric}}
\def\Rm{\text{\rm Rm}}

\def\p{\partial}

\def\p{\partial}

\def\p{\partial}

\def\KRF{K\"ahler-Ricci flow }

\def\bee{\begin{equation*}}
\def\eee{\end{equation*}}

\def\a{{\alpha}}
\def\b{{\beta}}

\def\p{\partial}

\def\K{K\"ahler }

\def\KRF{K\"ahler-Ricci flow }

\def\be{\begin{equation}}
\def\ee{\end{equation}}

\def\a{{\alpha}}
\def\b{{\beta}}

\def\Ric{\text{\rm Ric}}
\def\Rm{\text{\rm Rm}}

\def\p{\partial}

\def\p{\partial}

\def\p{\partial}

\def\KRF{K\"ahler-Ricci flow }

\def\cd{\nabla}

\def\bee{\begin{equation*}}
\def\eee{\end{equation*}}

\def\a{{\alpha}}
\def\b{{\beta}}

\def\p{\partial}

\def\K{K\"ahler }

\def\KRF{K\"ahler-Ricci flow }

\def\be{\begin{equation}}
\def\ee{\end{equation}}

\def\a{{\alpha}}
\def\b{{\beta}}

\def\Ric{\text{\rm Ric}}
\def\Rm{\text{\rm Rm}}

\def\p{\partial}

\def\p{\partial}

\def\p{\partial}

\def\KRF{K\"ahler-Ricci flow }

\def\ve{\varepsilon}
\def\inj{\textbf{Inj}}
\def\b{\beta}
\begin{document}
\title{ A note on existence of exhaustion functions and its applications}

\author{Shaochuang Huang}

\address{Yau Mathematical Sciences Center, Tsinghua University, Beijing, China.} \email{schuang@mail.tsinghua.edu.cn}

\thanks{Research partially supported by China Postdoctoral Science Foundation (The Tenth Special Grant)}

\begin{abstract} In this note, we prove an existence result on exhaustion functions by adapting the method in \cite{Tam10}. Then we apply it to prove short-time existence of Ricci flow and study Yau's uniformization conjecture using similar method in \cite{He16} and \cite{Lee-Tam17}. 


\end{abstract}

\maketitle
\markboth{Shaochuang Huang} {A note on existence of exhaustion functions and its applications}
\section{Introduction}

In this note, we prove an existence result of exhaustion functions using harmonic radius estimate, which adapts the method in \cite{Tam10}. Exhaustion functions are useful to study complete non-compact manifolds. In particular, one can use them to construct Ricci flows on such manifolds. See for example \cite{He16} and \cite{Lee-Tam17}. One classical result on existence of exhaustion functions is the following \cite{S-Y}:
\begin{thm}[Schoen-Yau] Let $M$ be a complete Riemannian manifold with $\Ric\geq-k$ for some $k\geq 0$, then there exists a smooth function $f$ defined on $M$ such that: \bee
|\cd f|<C,\hspace{.3cm} f\geq Cd(x),\hspace{.3cm} |\Delta f|<C. \eee Here $C$ is a constant depending on $M$ and $d(x)$ is the distance function with some fixed point.

\end{thm}

In some applications to study complete non-compact manifold, one wants to obtain a smooth exhaustion function with Hessian bound. In \cite{Shi}, W.-X. Shi proves the following:
\begin{thm}[W.-X. Shi]\label{Tam} Let $(M, g)$ be an $n$-dimensional complete Riemannian manifold with sectional curvature bounded by $K$ i.e. $|Sect(g)|\leq K$ for some $K\geq0$. Then there exists a smooth function $f$ defined on $M$ such that: \bee
d(x)+1\leq f(x)\leq d(x)+C, \eee $|\cd f|(x)\leq C$ and $|\cd\cd f|(x)\leq C$ for all $x\in M$. Here $C$ is a constant depending on $n, K$ and $d(x)$ is the distance function with some fixed point. 
\end{thm}

Later, L.-F. Tam \cite{Tam10} reproves the above result using heat equation together with harmonic coordinate. We adapt Tam's method to obtain the main result of this note: 

\begin{thm}\label{Ex} $\forall n\in\mathbb{N}$, there exist $\ve=\ve(n)>0$ and $\Lambda'(n)>0$ depending only on $n$ such that if $(M, g)$ is a smooth complete Riemannian $n$-manifold with $-\Lambda'\leq\Ric(g)\leq\Lambda$ and $V_gB_g(x, 1)\geq(1-\ve)\omega_n$ for all $x\in M$ and for some $\Lambda\geq 0$, then there exists a smooth function $\gamma: M\to\mathbb{R}$ satisfying \bee
d(x)+1\leq \gamma(x)\leq  d(x)+C \eee and \bee
|\cd \gamma|\leq C, |\cd\cd \gamma|\leq C. \eee Here $d(x)$ is the distance function with respect to $g$ from a fixed point $O\in M$ and $C$ is a constant depending only $n$ and $\Lambda$.

\end{thm}

In previous literature, it seems that if one wants to obtain an exhaustion function with Hessian bound, it needs to assume the bound of sectional curvature. We observe that this is not necessary.  The above theorem tells us one only needs to assume the bound of Ricci curvature and almost Euclidean volume condition. The proof of Theorem \ref{Ex} adapts the proof of Theorem \ref{Tam} by Tam \cite{Tam10}, uses harmonic radius estimate in \cite{Anderson90} and  observes that it is not necessary to use exponential map which is used in \cite{Tam10} .

Once we obtain Theorem \ref{Ex}, by the method in \cite{He16} and \cite{Lee-Tam17}, we can obtain a short-time existence result of Ricci flow:

\begin{thm}\label{main} Given any $n\in\mathbb{N}$, $\alpha>0$, $\Lambda\geq0$ and $r_0>0$ there exist $T(n,\alpha,\Lambda,r_0)>0$ depending only on $n$, $\alpha$, $\Lambda$ and $r_0$ and $\ve(n,\alpha)>0$ depending only on $n$, $\alpha$ such that if $(M, g)$ is a smooth complete Riemannian $n$-manifold with $|\Ric(g)|\leq\Lambda$ and $V_g B_g(x,r)\geq(1-\ve)\omega_nr^n$ for all $0<r\leq r_0$ and for all $x\in M$, then there exists a smooth complete Ricci flow $g(t)$ on $M\times[0, T]$ with $g(0)=g$ and \bee
\sup\limits_{x\in M} |\Rm|(x,t)\leq \frac{\alpha}{t} \eee for all $t\in(0, T]$.

\end{thm}

\begin{rem}  Our motivation for the above theorem is to remove the assumption of sectional curvature lower bound in \cite{Lee-Tam17}, but as a compensation, we have to assume the upper bound of Ricci curvature. Under the assumption of Ricci curvature as in Theorem \ref{Ex}, by \cite{Anderson90}, almost Euclidean isoperimetric inequality is equivalent to almost Euclidean volume condition. Therefore, our result is weaker than the result in \cite{He16}. The main difference from \cite{He16} is that we consider another pseudolocality theorem which is proved by Tian-Wang \cite{Tian-Wang15} instead of the original Perelman's one \cite{Perelman} which is used in \cite{He16}.

\end{rem}

We can also apply Theorem \ref{main} to study Yau's uniformization conjecture as in \cite{He16} and \cite{Lee-Tam17}. We can obtain the following:

\begin{cor}\label{unif}  Given any $n\in\mathbb{N}$, $\Lambda\geq0$ and $r_0>0$ there exists $\ve(n)>0$ depending only on $n$ such that if $(M, g)$ is a smooth complete non-compact \K $n$-manifold with nonnegative bisectional curvature, $\Ric(g)\leq\Lambda$ and maximum volume growth such that $V_g B_g(x,r_0)\geq(1-\ve)\omega_nr_0^{2n}$  for all $x\in M$, then $M$ is biholomorphic to $\mathbf{C}^n$.
\end{cor}

\begin{rem} The above result is not new and needs more assumptions. Actually, Gang Liu \cite{Liu} and Lee-Tam \cite{Lee-Tam172} independently prove that any complete non-compact \K manifold with complex dimension $n$, nonnegative bisectional curvature and maximum volume growth must be biholomorphic to $\mathbf{C}^n$. The argument here seems much simpler. 
\end{rem}

{\it Acknowledgement:} The author would like to thank Pak-Yeung Chan for some helpful discussions and thank Professor Luen-Fai Tam, Guoyi Xu, Fei He and Man-Chun Lee for their interests in this work and some useful comments.

\section{Harmonic radius estimate}

In this section, we will give a proof of a harmonic radius estimate which appears in Theorem 3.2 and its remark in \cite{Anderson90}. For more details on harmonic coordinate and harmonic radius, one may see \cite{DeTurck-Kazdan}, \cite{Anderson90} and \cite{Hebey-Herzlich}.

Now, we will show the following $\mathcal{C}^{1, \a}$-harmonic radius estimate:

\begin{thm}\label{C1a}  $\forall n\in\mathbb{N}$, there exist $\ve=\ve(n)>0, \Lambda'=\Lambda'(n)>0$ and $\delta=\delta(n)>0$ such that if $(M, g)$ is a smooth complete Riemannian $n$-manifold with $-\Lambda'\leq\Ric(g)\leq\Lambda$ and $V_gB_g(x_0, 1)\geq(1-\ve)\omega_n$ for some $x_0\in M$ and some $\Lambda\geq 0$, then $\forall Q>1, \a\in(0,1)$, there exists a constant $C=C(n, \Lambda, Q, \a)>0$ such that for all $x\in B_g(x_0, \delta)$ \bee
r_H(Q, 1, \a)(x)\geq C. \eee Here $r_H(Q, k, \a)(x)$ denotes the $\mathcal{C}^{k, \a}$-harmonic radius at $x$ i.e. the largest number such that on $B_x(r_H)$, there exists a harmonic coordinate such that \bee
Q^{-1}\delta_{ij}\leq g_{ij}\leq Q\delta_{ij} \eee and \bee 
\sum\limits_{1\leq |\b|\leq k} r_H^{\b}\sup\limits_x|\p^\b g_{ij}(x)|+\sum\limits_{|\b|=k}r_H^{k+\a}\sup\limits_{y\neq z} \frac{|\p^\b g_{ij}(y)-\p^\b g_{ij}(z)|}{d_g(y,z)^\a}\leq Q-1. \eee 
\end{thm}

By Sobolev embedding theorem, to prove Theorem \ref{C1a}, it suffices to show the following $\mathcal{H}^p_2$-harmonic radius estimate:

\begin{thm} $\forall n\in\mathbb{N}$, there exist $\ve=\ve(n)>0, \Lambda'=\Lambda'(n)>0$ and $\delta=\delta(n)>0$ such that if $(M, g)$ is a smooth complete Riemannian $n$-manifold with $-\Lambda'\leq\Ric(g)\leq\Lambda$ and $V_gB_g(x_0, 1)\geq(1-\ve)\omega_n$ for some $x_0\in M$ and some $\Lambda\geq 0$, then $\forall Q>1, p>n$, there exists a constant $C=C(n, \Lambda, Q, p)>0$ such that for all $x\in B_g(x_0, \delta)$ \bee
r_H(Q, 2, p)(x)\geq C. \eee Here $r_H(Q, k, p)(x)$ denotes the $\mathcal{H}^p_k$-harmonic radius at $x$ i.e. the largest number such that on $B_x(r_H)$, there exists a harmonic coordinate such that \bee
Q^{-1}\delta_{ij}\leq g_{ij}\leq Q\delta_{ij} \eee and \bee 
\sum\limits_{1\leq |\b|\leq k} r_H^{2-\frac{n}{p}}\|\p_\b g_{ij}\|_{L^p}\leq Q-1. \eee 
\begin{proof} We argue by contradiction. For some $n\in \mathbb{N}, Q>1, p>n, \Lambda\geq0$, there exist a sequence of smooth complete Riemannian $n$-manifolds $(M_m, g_m)$, a sequence of points $\overline{x_m}\in M$ and a sequence of points $x_m\in B_{g_m}(\overline{x_m}, \delta)$ such that for all $x\in M_m$, \bee
-\Lambda'\leq\Ric_{g_m}(x)\leq\Lambda \eee and \bee
V_{g_m}B_m\geq (1-\ve)\omega_n, \eee where $\ve, \Lambda'$ and $\delta$ will be chosen later depending only on $n$  and $B_m:=B_{g_m}(\overline{x_m}, 1)$. Moreover, $r_H(g_m, x_m)\to 0$. 

Now we consider the continuous function \bee x\mapsto \frac{r_H(g_m, x)}{d_{g_m}(x,\p B_{g_m}(\overline{x_m}, 2\delta))}\eee on $B_{g_m}(\overline{x_m}, 2\delta)$. 

Then there exists $y_m\in B_{g_m}(\overline{x_m}, 2\delta)$ such that for all $x\in B_{g_m}(\overline{x_m}, 2\delta)$, \bee
 \frac{r_H(g_m, y_m)}{d_{g_m}(y_m,\p B_{g_m}(\overline{x_m}, 2\delta))}\leq  \frac{r_H(g_m, x)}{d_{g_m}(x,\p B_{g_m}(\overline{x_m}, 2\delta))}. \eee In particular, \bee
 r_H(g_m, y_m)\leq \frac{r_H(g_m, x_m)}{d_{g_m}(x_m,\p B_{g_m}(\overline{x_m}, 2\delta))}\cdot d_{g_m}(y_m,\p B_{g_m}(\overline{x_m}, 2\delta))\leq \frac{2\delta}{2\delta-\delta}r_H(g_m, x_m)\to 0.\eee
 
 Therefore, we have \bee
 \lim\limits_{m\to\infty}  r_H(g_m, y_m)= \lim\limits_{m\to\infty}  \frac{r_H(g_m, y_m)}{d_{g_m}(y_m,\p B_{g_m}(\overline{x_m}, 2\delta))}=0. \eee
 
 Next, we consider $h_m:=r_H(g_m, y_m)^{-2}g_m$. Then we have \bee\begin{split}
&r_H(h_m, y_m)=1,\\
&\lim\limits_{m\to\infty} |\Ric_{(M_m, h_m)}|=0,\\
&\lim\limits_{m\to\infty} d_{h_m}(y_m, \p B_{g_m}(\overline{x_m}, 2\delta))=\infty.\end{split} \eee

Note that for all $y\in B_{g_m}(\overline{x_m}, 2\delta)$, \bee
r_H(h_m, y)=\frac{r_H(g_m, y)}{r_H(g_m, y_m)}\geq \frac{d_{g_m}(y,\p B_{g_m}(\overline{x_m}, 2\delta))}{d_{g_m}(y_m,\p B_{g_m}(\overline{x_m}, 2\delta))}=\frac{d_{h_m}(y,\p B_{g_m}(\overline{x_m}, 2\delta))}{d_{h_m}(y_m,\p B_{g_m}(\overline{x_m}, 2\delta))}.\eee Set $\eta_m=\frac{1}{d_{h_m}(y_m,\p B_{g_m}(\overline{x_m}, 2\delta))}$. Then $\lim\limits_{m\to\infty} \eta_m=0$ and for all $y\in B_{y_m}(\frac{1}{2\eta_m})$, \bee r_H(h_m, y)\geq \frac 1 2.\eee

Therefore, we obtain that given any $R>>1$, $r_H(h_m, y)\geq \frac 1 2$ on $B_{y_m}(R)$ for all $m$ large enough. Then we follow similar argument as Lemma 2.1 and Lemma 2.2 in \cite{Anderson90} (See also Theorem 11 and Proposition 12 in \cite{Hebey-Herzlich}), we see that $(B_{g_m}(\overline{x_m}, \delta), y_m, h_m)$ converges $\mathcal{H}^p_2$ to a limiting complete Riemannian $n$-manifold $(M, y, h)$ uniformly on compact set. Moreover, $(M, h)$ is smooth Ricci-flat complete Riemannian $n$-manifold.

Next, we will claim that $(M, h)$ is isometric to $\mathbb{R}^n$. We first claim that $V_hB_h(y, R)\geq (1-\ve')\omega_n R^n$ for all $R$ large enough, where $\ve'$ is a constant depending only on $n$ in the following lemma, i.e. Lemma \ref{isom}. To see this, we note that for all $r\leq 1-2\delta$, by volume comparison, we have \bee
\frac{V_{g_m}B_{g_m}(y_m, r+2\delta)}{\int^r_0 \sinh^{n-1}(\sqrt{\Lambda'}r)}\geq \frac{V_{g_m}B_{g_m}(\overline{x_m}, 1)}{\int^1_0 \sinh^{n-1}(\sqrt{\Lambda'}r)}\geq \frac{(1-\ve)\omega_n}{\int^1_0 \sinh^{n-1}(\sqrt{\Lambda'}r)}, \eee
which implies \bee
\frac{V_{h_m}B_{h_m}(y_m, r_m^{-1}(r+2\delta))}{\int^{r_m^{-1}(r+2\delta)}_0 \sinh^{n-1}(\sqrt{\Lambda'r_m}r)}\geq\frac{(1-\ve)\omega_n\int^r_0 \sinh^{n-1}(\sqrt{\Lambda'}r)}{r_m^{n}\int^1_0 \sinh^{n-1}(\sqrt{\Lambda'}r)\cdot \int^{r_m^{-1}(r+2\delta)}_0 \sinh^{n-1}(\sqrt{\Lambda'r_m}r)}. \eee Here $r_m:=r_H(g_m, y_m)$.

Now we consider the ball $B_h(y, 10R)$ in $M$ for any fixed large $R$. Let $\phi_m$ be the diffeomorphism from $U_m$ into $B_{g_m}(\overline{x_m}, \delta)$, where $\{U_m\}$ exhausts $M$. Then for any $\delta'>0$, there exists a $m_0$ may depending on $\delta'$ and $R$, for all $m\geq m_0$, $(1+\delta')^{-2}h\leq\phi^\ast_m h_m\leq(1+\delta')^2h$ on $B_h(y, 10R)$. Next we will claim for any $\delta''>0$, there exists a $m_1$ may depending on $\delta''$ and $R$, for all $m\geq m_1$, \bee
B_{h_m}(y_m, (1+\delta'')^{-1}R)\subset \phi_m(B_h(y, R)).\eee

To see this, we consider any $z\in B_{h_m}(y_m, (1+\delta'')^{-1}R)$, let $\gamma_m$ be the minimal geodesic from $y_m$ to $z$ with respect to $h_m$. Assume $\gamma_m$ can be pulled back by $\phi_m$ into $B_h(y, 10R)$ , then we have \bee 
d_h(\phi^{-1}_m(z), y)\leq (1+\delta')d_{h_m}(z, y_m)\leq\frac{1+\delta'}{1+\delta''}R.\eee Now one can take $\delta'=\delta''$, then we obtain \bee
B_{h_m}(y_m, (1+\delta'')^{-1}R)\subset \phi_m(B_h(y, R)).\eee Then we take $r=1-3\delta$, by volume comparison, it is not hard to see $V_hB_h(y, R)\geq (1-\ve')\omega_n R^n$ for all $R$ large enough, where we can choose $\delta, \ve, \Lambda'>0$ small enough depending on $n$. Therefore, by Lemma \ref{isom}, we see that $(M, h)$ is isometric to $\mathbb{R}^n$. This leads a contradiction on the harmonic radius.

It remains to show \bee
a:=\sup\{\eta\in[0,1] | \phi_m^{-1}(\gamma_m|_{[0,\eta]})\subset B_h(y, 10R)\}=1. \eee It is easy to see $a>0$ and we argue by contradiction. Assume $a<1$, then there exists $\eta_0>0$ such that $\gamma_m(\eta_0)\in \phi_m(\p B_h(y, 10R))$ and $\gamma_m|_{[0, \eta_0)}\subset \phi_m(B_h(y, 10R))$. It is not hard to see that $\gamma_m(\eta_0)\in  \phi_m(B_h(y, R))$, therefore one can extend $\gamma_m$ outside $\phi_m(B_h(y, R))$ a little which is in  $\phi_m(B_h(y, 10R))$. This leads a contradiction.

\end{proof}

\end{thm}

\begin{lma}\label{isom} For any $n\in\mathbb{N}$, there exists $\ve=\ve(n)>0$ such that if $(N, g)$ is a smooth complete Ricci-flat $n$-manifold with $V_gB_g(x_0, R)\geq (1-\ve)\omega_n R^n$ for some $x_0\in N$ for all $R$ large enough, then $N$ is isometric to $\mathbb{R}^n$.

\end{lma}

\begin{proof} We argue by contradiction. Suppose not, then there exists a sequence of  smooth complete Ricci-flat $n$-manifolds $(N_i, g_i)$ with $V_{g_i}B_{g_i}(R)/R^n\geq (1-\ve_i)\omega_n$, where $\ve_i\to0$ and $(N_i, g_i)$ is not isometric to $\mathbb{R}^n$ for all $i$. Now we choose an origin $x_0\in N_i$ for each $i$, and let $R_i=\inj_{N_i}(x_0)$ which is finite by the fact that $(N_i, g_i)$ is not isometric to $\mathbb{R}^n$. Let $x_i\in N_i$ be points realizing the minimum of \bee
c_i(x)=\frac{\inj_{N_i}(x)}{d_{N_i}(x, \p B_{x_0}(iR_i+i))}, \eee for $x\in B_{x_0}(iR_i+i)$. Then $\min c_i\to 0$ as $i\to\infty$. Now we consider $h_i:=(\inj_{N_i}(x_i))^{-2}g_i$, it follows that $\inj_{h_i}(x_i)=1$ and \bee
d_{h_i}(x_i, \p B_{x_0}(iR_i+i))\to \infty. \eee Therefore, as before, it follows that $\inj_{h_i}(x)$ is uniformly bounded below on balls of any fixed finite radius about $x_i$ in the $h_i$ metric. Then by Lemma 2.1 and Lemma 2.2 in \cite{Anderson90} (See also Theorem 11 and Proposition 12 in \cite{Hebey-Herzlich}), we see that $(N_i, x_i, h_i)$ sub-converges to a smooth complete Ricci-flat $n$-manifold $(Y, x, h)$ in the pointed $\mathcal{C}^\infty$ topology with $x=\lim x_i$. In particular, we have $\inj_h(x, Y)=1$. On the other hand, by letting $i\to\infty$, we have \bee
\frac{V_YB(r)}{r^n}\geq \omega_n, \eee which implies $Y$ is isometric to $\mathbb{R}^n$. This leads a contradiction.

\end{proof}

\begin{rem} By the similar argument as above, one can show Theorem \ref{C1a} under the assumption $|\Ric|\leq \Lambda$ and $V_{g}B_g(x_0, r_0)\geq (1-\ve(n))\omega_nr_0^n$ with $r_0$ and $\delta$ depending on $n$ and $\Lambda$.

\end{rem}

\section{Existence of Exhaustion functions and its applications}

In this section, we will give a proof of an existence result of exhaustion functions and its applications to short-time existence of Ricci flow and Yau's uniformization conjecture.

Once we obtain the harmonic radius lower bound from the previous section, we can construct an exhaustion function with bounded gradient and Hessian using the method as in Tam \cite{Tam10}. Namely, we can obtain our main result of this note, i.e. Theorem \ref{Ex}.

\textit{Proof of Theorem \ref{Ex}:} Since the $\mathcal{C}^0$ and gradient bound of the exhaustion function use only the completeness of the metric and the lower bound of Ricci curvature in the proof of \cite{Tam10}, we suffice to point out the difference in the proof of Hessian bound. Now by Theorem \ref{C1a}, we have harmonic radius lower bound. Then we use this harmonic coordinate chart instead of the exponential map in the original proof and do the computation in the Euclidean ball $B_e(0, a)$ where $a$ depends only on the harmonic radius lower bound, we can obtain the Hessian bound of the  exhaustion function following Tam's method. \qed

\begin{rem} Actually, to obtain an exhaustion function with bounded gradient and Hessian, it suffices to assume the metric is complete with Ricci curvature bounded from below and has a $\mathcal{C}^{1, \a}$-harmonic radius lower bound. 
\end{rem}

Now we can apply Theorem \ref{Ex} to show Theorem \ref{main}. Let $U_\rho:=\gamma^{-1}([0, \rho))$, where $\gamma$ is from Theorem \ref{Ex}. Then $U_\rho$ is pre-compact and exhausts $M$ as $\rho\to\infty$. We will modify $g$ in $U_\rho$ to obtain a smooth complete Riemannian metric $h$ with bounded curvature in $U_\rho$ which also preserves lower bound of Ricci curvature and volume ratio of $g$. We will show the following key lemma.

\begin{lma}\label{conf}  For all $n\in\mathbb{N}, r_0\geq 1$ and $\Lambda\geq 0$, there exist $\rho_0(n, \Lambda)>0, \kappa(n)>0, L(n, \Lambda)>0$ and $r_1(n, r_0, \Lambda)>0$ with the following property. Suppose $-\Lambda'\leq\Ric(g)\leq\Lambda$ and $V_gB_g(x, r)\geq (1-\ve)\omega_nr^n$ for all $x\in M$, $r\leq r_0$. Here $\ve$ and $\Lambda'$ are in Theorem \ref{Ex}. Then for all $\rho\geq\rho_0$, there exists a Riemannian metric $h$ on $U_\rho$ satisfies the following properties:

(i) $(U_\rho, h)$ is a smooth complete Riemannian manifold with bounded curvature and $h\equiv g$ on $U_{\rho(1-\kappa)}$.

(ii) $\Ric(h)\geq -L$.

(iii) For all $0\leq r\leq r_1$ with $B_h(x, r)\subset U_\rho$, $V_hB_h(x, r)\geq (1-2\ve)\omega_nr^n$.

\begin{proof} For $0<\kappa<1$ which will be determined later, we define a function $f:[0, 1)\to[0,\infty)$ by:\bee
f(s):=\left\{\begin{array}{ll}
0, & 0\leq s\leq 1-\kappa \\
-\log (1-(\frac{s-1+\kappa}{\kappa})^2), & 1-\kappa\leq s< 1
\end{array} \right. \eee

For $1-\kappa\leq s< 1$, we have:\bee
0<f'(s)=\frac{2(s-1+\kappa)}{\kappa^2-(s-1+\kappa)^2}\leq \frac{2\kappa}{(\kappa^2-(s-1+\kappa)^2)^2} \eee and \bee
0<f''(s)=\frac{2(\kappa^2+(s-1+\kappa)^2)}{(\kappa^2-(s-1+\kappa)^2)^2}\leq \frac{4\kappa^2}{(\kappa^2-(s-1+\kappa)^2)^2}.\eee

Now we consider a function on $U_\rho$ defined by \bee
f(x):=f(\rho^{-1}\gamma(x)). \eee Define $h:=e^{2f}g$, then $h$ is complete and $h\equiv g$ on $U_{\rho(1-\kappa)}$. Now we will show $h$ has bounded curvature. To see it, let ${e_i}$ be an orthonormal frame with respect to $g$, we have:\bee
K^h_{ij}=e^{-2f}(K^g_{ij}-\sum\limits_{k\neq i,j}|\cd_k f|^2+\cd_i\cd_i f+\cd_j\cd_j f). \eee

Note that \bee
e^{-2f}|\cd_if|^2\leq \rho^{-2}C_1^2e^{-2f}(\frac{2\kappa}{(\kappa^2-(s-1+\kappa)^2)^2} )^2\leq \frac{4C_1^2}{\rho^2\kappa^2} \eee and \bee
e^{-2f}|\cd_i\cd_if|\leq \frac{4C_1^2}{\rho^2\kappa^2}+\frac{2C_1}{\rho\kappa}. \eee Here $C_1$ is the constant from Theorem \ref{Ex} depending only on $n$ and $\Lambda$.

Therefore, we have $|K^h_{ij}|\leq C(U_\rho, g, \rho, \kappa, C_1)$.

Next, we will show a lower bound of Ricci curvature of $h$. To see it, we note that \bee
R^h_{ij}=R^g_{ij}-(n-2)\cd_i\cd_j f+(n-2)\cd_if\cd_jf-(\triangle f+(n-2)|\cd f|^2)g_{ij}.\eee

From this, it is easy to see $\Ric(h)\geq -L(n, \rho, \kappa, C_1)h$.

Finally, we will show a lower bound of volume ratio with respect to $h$ where we will determine $\kappa$.

(a) If $x\in U_{\rho(1-2\kappa)}$, then $B_g(x, 1)\subset U_{\rho(1-\kappa)}$. To see it, we consider the closest (with respect to $g$) two points $x\in\p U_{\rho(1-2\kappa)}$ and $y\in U_{\rho(1-\kappa)}$, we have\bee\begin{split}
d(x,y)&\geq d(O, y)-d(O, x)\\
&\geq \gamma(y)-C_1-\gamma(x)\\
&=\kappa\rho-C_1.
\end{split} \eee

Then for any $\rho\geq \frac{C_1+1}{\kappa}$, we have $d(x,y)\geq 1$. Thus, we obtain $B_g(x, 1)\subset U_{\rho(1-\kappa)}$. Since $h\equiv g$ on $U_{\rho(1-\kappa)}$, we have \bee
V_hB_h(x, r)\geq (1-\ve)\omega_nr^n \eee for all $r\leq 1$ and for all $x\in U_{\rho(1-2\kappa)}$.

(b) If $\gamma(x)\geq (1-2\kappa)\rho$, we first claim the following statement is true (for the proof, see \cite{He16} or \cite{Lee-Tam17}).

Claim: For all $1-2\kappa\leq s<1$, there exists $\tau>0$ such that $0<s-\tau<s+\tau<1$, \bee
1\leq \exp(f(s+\tau)-f(s-\tau))\leq 1+C_2\kappa \eee and \bee
\tau\exp(f(s-\tau))\geq C_3\kappa^2, \eee where $C_2$ and $C_3$ are some absolute constants. 

Let $s:=\frac{\gamma(x)}{\rho}$ and $r_2$ be the largest $r$ such that \bee
B_h(x, r)\subset\{ y\in U_\rho | s-\tau< \frac{\gamma(y)}{\rho} < s+\tau\}. \eee

Say for example, there is $y\in\p B_h(x, r_2)$ with $\gamma(y)=(s+\tau)\rho$. Let $(\xi(t), t\in[0,1])$ be a minimal geodesic joining $x$ and $y$ with respect to $h$ with length $l$ mesured by $g$. Then we have $r_2\geq \exp(f(s-\tau))l$. On the other hand, \bee\begin{split}
\tau\rho&=|\gamma(x)-\gamma(y)|\\
&=\int^1_0\langle \cd^g\gamma, \xi' \rangle_g dt\\
&\leq C_1l.\end{split} \eee

This implies $r_2\geq \frac{C_3\rho\kappa^2}{C_1}$ by the above claim. Hence for $r<\min\{r_0, r_2\}$, we have\bee\begin{split}
V_hB_h(x, r)&\geq V_hB_g(x, e^{-f(s+\tau)r})\\
&\geq \exp[n(f(s-\tau)-f(s+\tau))](1-\ve)\omega_nr^n\\
&\geq(1+C_2\kappa)^{-n}(1-\ve)\omega_nr^n\\
&\geq(1-2\ve)\omega_nr^n. \end{split} \eee Here we take $\kappa$ small enough depending only on $n$.

To sum up, we can take $\rho_0, L$ depending only on $n, \Lambda$ and $r_1$ depending only on $n, \Lambda, r_0$ satisfying (i), (ii), (iii) in the conclusion of the lemma.

\end{proof}

\end{lma}

\begin{rem} Although the function $f$ defined from Lemma \ref{conf} is not smooth at $s=1-\kappa$, one can smooth it as in \cite{Lee-Tam17}.

\end{rem}

Next, let us recall a pseudolocality theorem which is proved by Tian-Wang \cite{Tian-Wang15}.

\begin{thm}[Tian-Wang]\label{pseu} Given any $n\in\mathbb{N}$ and $\alpha>0$, there exists $\ve(n,\alpha)>0$ depending only on $n$ and $\alpha$ such that if $(M, g(t), t\in[0,T))$ is a smooth complete Ricci flow with bounded curvature and if $V_{g(0)} B_{g(0)}(x_0,1)\geq(1-\ve)\omega_n$ and $\Ric (g(0))\geq -\ve$ on $B_{g(0)}(x_0,1)$ for some $x_0\in M$, then for all $t\in[0, T)\cap(0,\ve]$, we  have \bee
|\Rm|(x_0,t)\leq \frac{\a}{t}. \eee

\end{thm}

Then we can give a lower bound for the maximal existence time for Ricci flow by the above pseudolocality theorem.

\begin{lma}\label{maximal time}  Given any $n\in\mathbb{N}$ and $\alpha>0$, let $\ve(n,\alpha)>0$ be as in the above theorem and $(M, g(t), t\in[0,T))$ be a smooth complete Ricci flow with bounded curvature, where $T$ is the maximal existence time. Suppose  $V_{g(0)} B_{g(0)}(x,1)\geq(1-\ve)\omega_n$ and $\Ric (g(0))\geq -\ve$ on $B_{g(0)}(x,1)$ for all $x\in M$. Then $T\geq \ve$.

\end{lma}

\begin{proof} If $T=\infty$, then the result holds trivially. If $T<\infty$, then $\limsup\limits_{t\to T^-}\sup\limits_{M}|\Rm|=\infty$. Now we argue by contradiction. Suppose $T<\ve$, then by the above pseudolocality theorem, we have for all $x\in M$, $|\Rm|(x,t)\leq \frac{\a}{t}$. This leads a contradiction. 

\end{proof}

Now we can prove Theorem \ref{main}:

\textit{Proof of Theorem \ref{main}:} By rescaling the metric $g$ and applying Lemma \ref{conf}, for each $U_\rho$, there exists a Riemannian metric $h_\rho$ on it such that (i), (ii) and (iii) in Lemma \ref{conf} hold. By Shi's short-time existence of Ricci flow \cite{Shi89}, we obtain a smooth complete Ricci flow $h_\rho(t)$ with bounded curvature on $U_\rho$ and $h_\rho(0)=h_\rho$. By rescaling $h_\rho$ and applying Theorem \ref{pseu}, Lemma \ref{maximal time}, we see that for all $\rho$ large, $h_\rho(t)$ exists on $[0, T(n,\a,\Lambda,r_0)]$ and \bee\sup\limits_{x\in U_\rho}|\Rm(h_\rho(t))|(x)\leq \frac{\a}{t}.\eee Then one can use Chen-Simon's result (\cite{Chen09} and \cite{Simon12}) and modified Shi's estimate \cite{Lu-Tian11}  to obtain a complete Ricci flow on $M\times[0, T]$. To see the completeness, one can use Perelman's distance distortion estimate \cite{Perelman}.  This completes the proof of Theorem \ref{main}. \qed

Finally, let us show Corollary \ref{unif}:

\textit{Proof of Corollary \ref{unif}:} By Theorem \ref{main}, we obtain a smooth complete Ricci flow with curvature bounded by $\frac{\a}{t}$ in a short time. Then by \cite{Huang-Tam} or \cite{Lee-Tam17}, we see that $g(t)$ is also \K and has nonnegative bisectional curvature. By \cite{Chau-Tam}, to prove $M$ is biholomorphic to $\mathbf{C}^n$, it suffices to show $g(t)$ has maximum volume growth. By \cite{He16} (or by Corollary 6.2 in \cite{Simon} with a rescaling arguement), we see this is true. Therefore, we complete the proof of Corollary \ref{unif}. \qed

\end{document}